\documentclass[12pt, reqno]{amsart}
\usepackage{amssymb,latexsym,amsmath,amscd,amsthm,graphicx, color}
\usepackage[all]{xy}
\usepackage{pgf,tikz}
\usepackage{mathrsfs}
\usetikzlibrary{arrows}
\usepackage[left=0.6 in, top=0.6 in, right=0.6 in, bottom=0.3 in]{geometry}
\makeatletter
\newcommand{\setword}[2]{%
  \phantomsection
  #1\def\@currentlabel{\unexpanded{#1}}\label{#2}%
}
\makeatother

\usepackage[linkcolor=blue, urlcolor=blue, citecolor=blue,
colorlinks, bookmarks]{hyperref}
\definecolor{uuuuuu}{rgb}{0.26666666666666666,0.26666666666666666,0.26666666666666666}
\definecolor{xdxdff}{rgb}{0.49019607843137253,0.49019607843137253,1.}
\definecolor{ffqqqq}{rgb}{1.,0.,0.}
\definecolor{ffqqqq}{rgb}{1.,0.,0.}
\definecolor{ffxfqq}{rgb}{1.,0.4980392156862745,0.}

\definecolor{uuuuuu}{rgb}{0.26666666666666666,0.26666666666666666,0.26666666666666666}
\definecolor{qqwuqq}{rgb}{0.,0.39215686274509803,0.}
\definecolor{zzttqq}{rgb}{0.6,0.2,0.}
\definecolor{xdxdff}{rgb}{0.49019607843137253,0.49019607843137253,1.}
\definecolor{qqqqff}{rgb}{0.,0.,1.}
\definecolor{cqcqcq}{rgb}{0.7529411764705882,0.7529411764705882,0.7529411764705882}
\definecolor{sqsqsq}{rgb}{0.12549019607843137,0.12549019607843137,0.12549019607843137}

\theoremstyle{plain}

\newtheorem{theorem}[subsection]{Theorem}

\newtheorem{lemma}[subsection]{Lemma}

\newtheorem{prop}[subsection]{Proposition}

\theoremstyle{definition}

\newtheorem{remark}[subsection]{Remark}

\newcommand{\uu}{\cup}% union
\newcommand{\ii}{\cap}% intersection
\newcommand{\ci}{\subseteq}% contained in with equality
\newcommand{\sci}{\subset}% strictly contained in
\newcommand{\es}{\emptyset}% the empty set
\newcommand{\set}[1]{\{#1\}}% set
\newcommand{\ga}{\alpha}

\newcommand{\gk}{\kappa}

\newcommand{\gn}{\nu}
\newcommand{\go}{\omega}

\newcommand{\gs}{\sigma}

\newcommand{\tit}{\textit}% text italic

\newcommand{\D}[1]{\mathbb{#1}}% Doubled - blackboard bold - only caps, use as \D{A}
\newcommand{\te}{\text}% same as \mathrm command.

\newcommand{\ol}{\overline}
\newcommand{\ul}{\underline}
\begin{document}
	
	To appear, Monatshefte f\"ur Mathematik 
	\title{Quantization dimensions for inhomogeneous bi-Lipschitz Iterated Function Systems}

	\author{Amit Priyadarshi}
	\address{Department of Mathematics, Indian Institute of Technology Delhi, New Delhi, India 110016, Indian Institute of Technology Delhi-Abu Dhabi, Abu Dhabi, United Arab Emirates}
	\email{priyadarshi@maths.iitd.ac.in}
	
	\author{Mrinal K. Roychowdhury}
	\address{School of Mathematical and Statistical Sciences, University of Texas Rio Grande Valley, 1201
		West University Drive, Edinburg, TX 78539-2999, USA}
	\email{mrinal.roychowdhury@utrgv.edu}
	
	\author{Manuj Verma}
	\address{Department of Mathematics, Indian Institute of Technology Delhi, New Delhi, India 110016}
	\email{mathmanuj@gmail.com}
	
	\date{}


\date{}
\maketitle
	\pagestyle{myheadings}\markboth{A. Priyadarshi, M. K. Roychowdhury and M. Verma }{Quantization dimensions for inhomogeneous bi-Lipschitz Iterated Function Systems}
	
	\begin{abstract}
		Let $\nu$ be a Borel probability measure on a $d$-dimensional Euclidean space $\mathbb{R}^d$, $d\geq 1$, with a compact support, and let $(p_0, p_1, p_2, \ldots, p_N)$ be a probability vector with $p_j>0$ for $0\leq j\leq N$. Let $\{S_j: 1\leq j\leq N\}$ be a set of contractive mappings on $\mathbb{R}^d$. Then, a Borel probability measure $\mu$ on $\mathbb R^d$ such that $\mu=\sum_{j=1}^N p_j\mu\circ S_j^{-1}+p_0\gn$ is called an inhomogeneous measure, also known as a condensation measure on $\mathbb{R}^d$. For a given $r\in (0, +\infty)$, the quantization dimension of order $r$, if it exists, denoted by $D_r(\mu)$, of a Borel probability measure $\mu$ on $\mathbb{R}^d$ represents the speed at which the $n$th quantization error of order $r$ approaches to zero as the number of elements $n$ in an optimal set of $n$-means for $\mu$ tends to infinity. In this paper, we investigate the quantization dimension for such a condensation measure.
	\end{abstract}

	\section{Introduction}\label{se1}
	Quantization is a process of discretization, in other words, to represent a continuous or a large set of values by a set with smaller number of values. It has broad application in signal processing and data compression. 
Although the quantization is known to electrical engineers for several decades, Graf and Luschgy first introduced it in Mathematics (see \cite {GL1}). Quantization dimension of a Borel probability measure gives the speed how fast the quantization error approaches to zero as the number of elements in an optimal set of $n$-means approaches to infinity.  Quantization dimension is also connected with other dimensions of dynamical systems; for example, the lower quantization dimension of a Borel probability measure lies between the Hausdorff dimension and the lower box counting dimension of the measure, and the upper quantization dimension lies between the packing dimension and the upper box counting dimension of the measure (see  \cite{DAi,Potz}). Quantization dimension has exactly been determined for some fractal probability measures. In the known cases it has been seen that if the quantization dimension of a Borel probability measure exists, then it has connection with temperature function of the thermodynamic formalism that arises in multifractal analysis of the measure (for example, one can see  
\cite{GL3, LM, R6, R5, R3, R4, R2, R1}). For some relevant work interested readers can also see \cite{ALSW, AS, MR, PRV, PVV, RS, S1}. Quantization theory has recently been extended to constrained quantization. After the introduction of constrained quantization, and then conditional quantization, the quantization theory is now much more enriched with huge applications in our real world (see \cite{BCDRV, PR1, PR2}). To see many other earlier work interested readers can consult \cite{AW, B, BW, GG, GL2,   GL4, GN, GKL, G, GL, Z1, Z2}. 
	
	Given a Borel probability measure $P$ on $\D R^d$, a number $r \in (0, +\infty)$ and a natural number $n \in \D N$, the $n$th \tit{quantization
		error} of order $r$ for $P$, is defined by
	\begin{equation}\label{eq1}  V_{n, r}:=V_{n,r}(P)=\te{inf}\set{\int d(x, \ga)^r d
			P(x) : \ga \sci \D R^d, \, \te{card}(\ga) \leq n},\end{equation}
	where $d(x, \ga)$ denotes the distance from the point $x$ to the set
	$\ga$ with respect to a given norm $\|\cdot\|$ on $\D R^d$.
	The numbers
	\begin{equation} \label{eq2} \ul D_r(P):=\liminf_{n\to \infty}  \frac{r\log n}{-\log V_{n,r}(P)} \te{ and } \ol D_r(P):=\limsup_{n\to \infty} \frac{r\log n}{-\log V_{n, r}(P)} 
	\end{equation}
	are called the \tit{lower} and the \tit{upper quantization dimensions} of $P$ of order $r$, respectively. If $\ul D_r (P)=\ol D_r (P)$, the common value is called the \tit{quantization dimension} of $P$ of order $r$ and is denoted by $D_r(P)$. Quantization dimension measures the speed at which the specified measure of the error goes to zero as $n$ tends to infinity.
	For any $\gk>0$, the numbers $\liminf_n n^{\frac r \gk}  V_{n, r}(P)$ and $\limsup_n  n^{\frac r \gk}V_{n, r}(P)$ are called the \tit{$\gk$-dimensional lower} and \tit{upper quantization coefficients} for $P$, respectively. If the $\gk$-dimensional lower and upper quantization coefficients are finite and positive, then $\gk$ equals the quantization dimension $D_r(P)$.

	Let  $\mathcal{F}=\{\mathbb{R}^d; S_1,S_2,\ldots,S_N\}$ be a family of contractive iterated function system (IFS) such that for all $x, y\in \D R^d$, we have
	$$s_id(x,y)\leq d(S_i(x),S_i(y))\leq c_id(x,y),$$
	where $0<s_i\leq c_i<1$ and $i\in\{1,2,\ldots,N\}.$ 
	Let $K_{\phi}$ be the unique attractor of the IFS $\mathcal{F}$, that is,
	$$K_\es=\bigcup_{i=1}^{N}S_i(K_\es) .$$
	Let $C\subset\mathbb{R}^d$ be a compact set. Let $\nu$ be a Borel probability measure supported on $C$ such that $\nu(C)=1.$
	Let $(p_0,p_1,p_2,\ldots,p_N)$ be a probability vector. Then, there exists a unique Borel probability measure $\mu$ and a compact set $K$ such that
	$$\mu=p_0\nu+\sum_{i=1}^{N}p_i\mu\circ S_i^{-1},\quad K=\bigcup_{i=1}^{N}S_i(K)\bigcup C.$$
	We call $(\{S_i\}_{i=1}^{N}, (p_i)_{i=0}^{N},\nu)$ a \tit{condensation system}. The measure $\mu$ is called the \tit{invariant measure} or the \tit{condensation measure} for $({S_i}_{i=1}^{N}, (p_i)_{i=0}^{N},\nu)$, and the set $K$, which is the support of the measure $\mu,$  is called the \tit{attractor} of the system (see \cite{BD, L}). Such a measure is also termed as an \tit{inhomogeneous measure}.  We say that the condensation system satisfies the \tit{strong separation condition} (SSC) if $S_1(K), S_2(K), \ldots, S_N(K)~~\text{and}~~C$ are pairwise disjoint. An IFS
	$\mathcal{F}=\{\mathbb{R}^d; S_1,S_2,\ldots,S_N\}$  satisfies the \tit{open set condition} (OSC) if there exists a bounded nonempty open set $U \sci \D R^d$ such that $\uu_{j=1}^N S_j(U) \ci U$ and $S_i(U) \ii S_j(U) = \es$ for all $1\leq i \neq j\leq N$. Furthermore, $\mathcal{F}$ satisfies the \tit{strong open set condition} (SOSC) if $U$ can be chosen such that $U \ii K_{\es} \neq \es$, where $\es$ is the empty set and $K_\es$ is the attractor of $\mathcal{F}$. Notice that in the case of an iterated function system consisting of a finite number of self-similar mappings or conformal mappings on the Euclidean space, the open set condition implies the strong open set condition (see \cite{PRSS, S}), but it is not true for an iterated function system with infinitely many mappings (see \cite{SW}).
	
	In this paper, we state and prove the following theorem which is the main result of the paper.
	\begin{theorem}\label{main}
		Let $r\in (0,  +\infty)$, and let $\mu$ be the condensation measure associated with the condensation system
		$(\{S_i\}_{i=1}^{N}, (p_i)_{i=0}^{N},\nu)$. Assume that the condensation system $(\{S_i\}_{i=1}^{N},(p_i)_{i=0}^{N},\nu)$ satisfies the strong separation condition. Then,
		$$\max\{k_r, \underline{D}_r(\nu)\}\leq \underline{D}_r(\mu),$$
		where $k_r$ is uniquely determined by $\sum_{i=1}^{N}(p_is_{i}^{r})^{\frac{k_r}{r+k_r}}=1.$ Let $\mathcal{G}=\{\mathbb{R}^d; g_1,g_2,\ldots,g_M\}$ be another IFS satisfying the strong open set condition such that
		$d(g_i(x),g_i(y))\leq r_id(x,y),$ where $0<r_i<1$ and $i\in \{1,2,\ldots,M\}.$ Furthermore, if the measure $\nu$ is an invariant measure corresponding to the IFS $\mathcal{G}$ with probability vector $(t_1,t_2,\ldots,t_M)$, i.e.,  $\nu=\sum_{i=1}^{M}t_i\nu\circ g_{i}^{-1}$, then
		$$ \overline{D}_r(\mu)\leq \max\{l_r,d_r\},$$
		where $l_r$ and $d_r$ are given by $\sum_{i=1}^{N}(p_ic_{i}^{r})^{\frac{l_r}{r+l_r}}=1$ and $\sum_{i=1}^{M}(t_ir_{i}^{r})^{\frac{d_r}{r+d_r}}=1$, respectively.
	\end{theorem}
	
	\begin{remark}
		In the above theorem, if we consider the mappings $S_i$ and $g_i$ as similarity mappings, i.e.,
		$$d(S_i(x),S_i(y))=s_id(x,y),\quad d(g_j(x),g_j(y))=r_jd(x,y)$$
		for all $i\in \{1,2,\ldots,N\}$ and $j\in \{1,2,\ldots,M\},$ then the quantization dimension of the condensation measure $\mu$ exists, and
		$$D_r(\mu)=\max\{k_r,D_r(\nu)\},$$
		where $k_r$ and $D_r(\nu)$ are given by  $\sum_{i=1}^{N}(p_is_{i}^{r})^{\frac{k_r}{r+k_r}}=1$ and $\sum_{i=1}^{M}(t_ir_{i}^{r})^{\frac{D_r(\nu)}{r+D_r(\nu)}}=1$, respectively.
	\end{remark}
	\begin{remark}
		In \cite{Z5}, Zhu estimated the quantization dimension of the condensation measure $\mu$ corresponding to the condensation system $(\{f_i\}_{i=1}^{N}, (p_i)_{i=0}^{N},\nu)$ with an inhomogeneous open set condition, where each $f_i$ is a similarity mapping and the measure $\nu$ is an invariant measure corresponding to some self-similar IFS satisfying the open set condition. For determining the upper and the lower bound, the author used the separation condition on the  condensation system as well as the separation condition on the self-similar IFS associated with the measure $\nu$. On the other hand, in Theorem \ref{main}, we note that for determining the lower bound, we only take separation condition on the condensation system  $(\{S_ i\}_{i=1}^{N}, (p_i)_{i=0}^{N},\nu)$ and the measure $\nu$ is any probability measure, and in obtaining the upper bound, we consider the condensation system with strong separation condition and the measure $\nu$ corresponding to some IFS with strong open set condition.
	\end{remark}
	\begin{remark}
		In \cite{Z3}, Zhu determined the quantization dimension of the condensation measure $\mu$ corresponding to the condensation system  	$(\{f_i\}_{i=1}^{N}, (p_i)_{i=0}^{N},\nu)$ with the strong separation condition, where $f_i$'s are similarity mappings and the measure $\nu$ is a self similar measure corresponding to the IFS $\{\mathbb{R}^d; f_1,f_2,\ldots,f_N\}$ with a probability vector $(t_1,t_2,\ldots,t_N).$  By taking the same conditions as in \cite{Z3}, Roychowdhury \cite{R} established a relationship between the upper bound of the quantization dimension and the temperature function of the thermodynamic formalism that arises in the multifractal analysis of $\mu.$  However, in Theorem \ref{main}, we consider more general condensation system  $(\{S_ i\}_{i=1}^{N}, (p_i)_{i=0}^{N},\nu)$  with the strong separation condition by taking the maps  $S_i$'s as bi-Lipschitz and the measure $\nu$ as the invariant  measure corresponding to some IFS.
		Thus, Theorem \ref{main} provides the quantization dimension of the condensation measure in a more general setting.
	\end{remark}
	
	\section{Preliminaries}\label{se2}
	
	Define $I:=\{1,2,\ldots,N\}$ and let $I^n=\{w= w_1w_2w_3\cdots w_n: w_i\in I~~ \forall ~~i\in \{1,2,\ldots,n\}\}.$ Let $I^*=\cup_{n\in \mathbb{N}\uu \set{0}}I^n$, where $I^0$ denotes the set consisting of the empty sequence $\es$.
	Thus, $I^\ast$ denotes the set of all finite sequences of symbols belonging to $I$ including the empty sequence $\emptyset$. Notice that if $\go \in I^\ast$, then $\go \in I^n$ for some $n\in \D N\uu \set{0}$. Such an $n$ is called the length of the word $\go$ and is denoted by $|\go|$. Notice that the length of the empty sequence is zero. Let $\Omega$ denote the set of all infinite sequences of symbols from the set $I.$  For $w=w_1w_2\cdots w_n\in I^n$ with $n\geq 1$, we define
	$$S_w=S_{w_1}\circ S_{w_2} \circ \cdots \circ S_{w_n},~~ p_w=\prod_{i=1}^{n}p_{w_i},~~ s_w=\prod_{i=1}^{n}s_{w_i},~~ c_w=\prod_{i=1}^{n}c_{w_i}~\text{and}~w^-=w_1w_2\cdots w_{n-1}.$$
	For the empty sequence $\es$, we write
	\[S_\es=Id_{\D R^d}, \, p_\es=1, \, \te{ and } s_\es=c_\es=1,\]
	where $Id_{\D R^d}$ is the identity mapping on $\D R^d$.
	For $n\geq 1$, by iterating, we obtain
	$$K=\bigg(\bigcup_{w\in I^n}S_w(K)\bigg)\cup \bigg(\bigcup_{m=0}^{n-1}\bigcup_{w\in I^m}S_w(C)\bigg),~\text{and}~ \mu=\sum_{w\in I^n}p_w\mu\circ S_{w}^{-1}+p_0\sum_{m=0}^{n-1}\sum_{w\in I^m}p_w \nu \circ S_{w}^{-1}.$$
	For any $w\in I^*$ such that $|w|\geq n$, we denote $w|_{n}=w_1w_2\cdots w_n.$ We define an ordering on $I^*$ as follows
	$$w\prec \sigma~~\text{if and only if}~~|w|\leq |\sigma|,~~ \sigma|_{|w|}=w,$$
	where $\sigma,w\in I^*.$  For any $\sigma, w\in I^*$, we say that $\sigma$ and $w$ are incomparable if neither $\sigma\prec w$ nor $w\prec \sigma$. On the other hand, if $\go\prec \gs$, we say that $\gs$ is an extension of $\go$. For $w\in I^*$, we define
	$$\Xi_{w}(m)=\{\sigma\in I^{|w|+m}: w\prec \sigma\}\quad \text{and}\quad \Xi_{w}^{*}=\cup_{m\in \mathbb{N}}\Xi_{w}(m),$$
	where $m\geq 1.$\\
	Let $\mathcal{G}=\{\mathbb{R}^d; g_1,g_2,\ldots,g_M\}$ be another IFS such that
	$$d(g_i(x),g_i(y))\leq r_id(x,y),$$ where $0<r_i<1$ and $i\in \{1,2,\ldots,M\}.$
	Define $J:=\{1,2,\ldots,M\}.$ We define $J^*$ and $J^n$ in a similar way as we define $I^*$ and $I^n.$
	\par
	We call $\Gamma\subset I^*$ a finite maximal antichain if $\Gamma$ is a finite set such that every sequence in $\Omega$ is an extension of some element in $\Gamma$, but no element of $\Gamma$ is an extension of another element in $\Gamma.$
	Let $\Gamma$ be a finite maximal antichain. Set
	$$l_\Gamma=\min\{|\omega|: \omega\in \Gamma\} \quad \text{and} \quad M_\Gamma=\max\{|\omega|: \omega\in \Gamma\}.$$
	For $w\in I^{l_\Gamma}$, we define
	$$\Delta_{\Gamma}(w)=\{\sigma\in I^*: w\prec \sigma,~~\Xi_{\sigma}^{*}\cap \Gamma\ne \emptyset \},\quad \Delta_{\Gamma}^{*}=\cup_{w\in I^{l_\Gamma}}\Delta_{\Gamma}(w).$$
	\begin{lemma} (see \cite{Z5}) \label{Z}
		Let $\Gamma$ be a finite maximal antichain. Then,
		$$K=\bigg(\bigcup_{w\in \Gamma}S_{w}(K)\bigg)\cup \bigg(\bigcup_{w\in \Delta_{\Gamma}^{*}} S_{w}(C)\bigg)\cup \bigg(\bigcup_{n=0}^{l_\Gamma-1}\bigcup_{w\in I^n}S_{w}(C)\bigg).$$
	\end{lemma}
	For a given Borel probability measure $P$ on $\D R^d$, a positive real number $r$ and a natural number $n$, let $V_{n, r}$ denote the $n$th quantization error of order $r$ as defined in equation \eqref{eq1}. If the infimum in \eqref{eq1} is attained for some $A  \sci \D R^d$, then $A$ is called an $n$-\emph{optimal set} for the measure $P$ of order $r$. The existence of $n$-optimal set is guaranteed if $P$ has a compact support (see \cite{GL1}). The lower quantization dimension $\underline{D}_r$ and the upper quantization dimension $\overline{D}_r$ are as defined in equation \eqref{eq2}. We recall the following important proposition which will be needed to prove our main results.
	
	\begin{prop}(see \cite{GL1})\label{Prop2.2}
		\begin{enumerate}
			\item If $0\leq t_1<\overline{D}_r<t_2,$ then
			$$\limsup_{n\to \infty}n V_{n,r}^{\frac{t_1}{r}}=\infty~~~\text{and}~~~\lim_{n\to \infty}n V_{n,r}^{\frac{t_2}{r}}=0.$$
			\item If $0\leq t_1<\underline{D}_r<t_2,$ then
			$$\liminf_{n\to \infty}n V_{n,r}^{\frac{t_2}{r}}=0~~~\text{and}~~~\lim_{n\to \infty}nV_{n,r}^{\frac{t_1}{r}}=\infty.$$
		\end{enumerate}
		
	\end{prop}
	\section{Main results}\label{se3}
	
	To prove Theorem~\ref{main}, which gives the main results of the paper, we state and prove some lemmas and propositions.  
	
	\begin{lemma} \label{lem3.1}
		Let $r\in (0,  +\infty)$, and let $\mu$ be the condensation measure associated with the condensation system
		$(\{S_i\}_{i=1}^{N}, (p_i)_{i=0}^{N},\nu)$. Then,
		$$\underline{D}_r(\nu)\leq \underline{D}_r(\mu)\quad \text{and}\quad \overline{D}_r(\nu)\leq \overline{D}_r(\mu).$$
	\end{lemma}
	\begin{proof}
		Let $A$ be an $n$-optimal set for the Borel probability measure $\mu$ of order $r$. Then, we have
		$$V_{n,r}(\mu)=\int d(x,A)^rd\mu(x).$$
		Since $\mu=p_0\nu+\sum_{i=1}^{N}p_i\mu\circ S_i^{-1}$, 	we obtain $\mu\geq p_0 \nu.$ Then,
		\begin{align*}
			V_{n,r}(\mu)\geq p_0\int d(x,A)^rd\nu(x)\geq p_0 V_{n,r}(\nu).
		\end{align*}
		In the above, taking logarithm, and then dividing both side by $\log V_{n,r}(\mu) \log V_{n,r}(\nu)$, we get
		$$\frac{1}{\log V_{n,r}(\mu)}\leq \frac{-\log p_0}{\log V_{n,r}(\mu) \log V_{n,r}(\nu)}+\frac{1}{\log V_{n,r}(\nu)},$$
		which implies that
		$$\bigg(1-\frac{\log p_0}{\log V_{n,r}(\mu)}\bigg)\frac{r \log n}{-\log V_{n,r}(\nu)}\leq \frac{r \log n}{-\log V_{n,r}(\mu)}.$$
		By \cite[Lemma 6.1]{GL1}, we have $\lim_{n\to \infty}V_{n,r}(\mu)\to 0.$ Hence, by taking liminf and limsup on both sides of the above inequality, we get our required result.
	\end{proof}
	In the following proposition, we only assume that the condensation system $(\{S_i\}_{i=1}^{N},(p_i)_{i=0}^{N},\nu)$ satisfies the strong separation condition without assuming any condition on the probability measure $\nu.$
	\begin{prop}\label{Prop3.2}
		Assume that the condensation system $(\{S_i\}_{i=1}^{N},(p_i)_{i=0}^{N},\nu)$ satisfies the strong separation condition. Let $\mu$ be the inhomogeneous measure associated with   $(\{S_i\}_{i=1}^{N},(p_i)_{i=0}^{N},\nu)$. Then for any $r>0$, we have 
		$$\liminf_{n\to \infty} n V_{n,r}^{\frac{k_r}{r}}(\mu)>0,$$
		where $k_r$ is uniquely determined by $\sum_{i=1}^{N}(p_is_{i}^{r})^{\frac{k_r}{r+k_r}}=1.$
	\end{prop}
	\begin{proof}
		Let $r\in (0,  +\infty)$. We define a mapping $\phi: [0,+\infty)\to [0,+\infty)$ by
		$$\phi(t)=\sum_{i=1}^{N}(p_is_{i}^{r})^{t}.$$
		The map $\phi$ is strictly decreasing and continuous. Moreover, $\phi(0)=N>1$ and $\phi(1)=\sum_{i=1}^{N}(p_is_{i}^{r})<\sum_{i=1}^{N}p_i<1.$ Thus, there exists a unique $\Tilde{t_r}\in (0,1)$ such that $\sum_{i=1}^{N}(p_i s_{i}^{r})^{\Tilde{t_r}}=1.$  By the uniqueness of $\Tilde{t_r}$, there is a unique $k_r\in (0,  +\infty)$ such that $$\sum_{i=1}^{N}(p_is_{i}^{r})^{\frac{k_r}{r+k_r}}=1.$$ 
		Since the condensation system $(\{S_i\}_{i=1}^{N},(p_i)_{i=0}^{N},\nu)$ satisfies the strong separation condition, $S_1(K),S_2(K),\ldots,S_N(K)~\text{and}~C$ are pairwise disjoint sets. Write
		$$\delta=\min\{d(S_i(K),S_j(K)): 1\leq i\ne j\leq N\}>0.$$
		Let $A_n$  be an $n$-optimal set for the Borel probability measure $\mu$ of order $r$. Set $\delta_n=\sup_{x\in K}d(x,A_n)$. Then, $\lim_{n\to \infty}\delta_n=0$ (see \cite[Lemma 6.1, Lemma 13.8]{GL1}). This implies that there exists $n_0\in \mathbb{N}$ such that $\delta_n\leq \frac{\delta}{3}$ for all $n\geq n_0.$ Set $A_{n_i}=\{a\in A_n: d(a,S_i(K))\leq \delta_n\}.$ It is easy to see that $A_{n_i}\ne \emptyset$ and $A_{n_i}\cap A_{n_j}=\emptyset$ for $1\leq i\ne j\leq N.$ Let $k_i(n)=\text{card}(A_{n_i}).$ Thus, we have $\sum_{i=1}^{N}k_i(n)\leq n.$ Also, $k_i(n)\leq n-1$ for each $1\leq i\leq N.$
		Since $A_n$ is an $n$-optimal set for the Borel probability measure $\mu$ of order $r$, we have
		\begin{align*}
			V_{n,r}(\mu)&=\int d(x,A_n)^r d\mu(x)\\&\geq \sum_{i=1}^{N} \int_{S_i(K)} d(x,A_n)^r d\mu(x)\\&=\sum_{i=1}^{N} \int_{S_i(K)} d(x,A_{n_i})^r d\mu(x)\\&\geq  \sum_{i=1}^{N}p_i \int_{S_i(K)} d(x,A_{n_i})^r d(\mu\circ S_{i}^{-1})(x)\\& \geq  \sum_{i=1}^{N}p_i s_{i}^{r}\int_{K} d(x,S_{i}^{-1}(A_{n_i}))^r d(\mu)(x)\\&\geq \sum_{i=1}^{N}p_i s_{i}^{r} V_{k_{i}(n),r}(\mu).
		\end{align*}
		Define $M_0= \min\{n^{\frac{r}{k_r}}V_{n,r}(\mu): n<n_0\}.$ Clearly, $M_0>0$ and $M_0\leq n^{\frac{r}{k_r}}V_{n,r}(\mu)$ for all $n<n_0$. Now our claim is that $M_0\leq n^{\frac{r}{k_r}}V_{n,r}(\mu)$  for all $n\in \mathbb{N}.$ We prove our claim by induction on $n\in \mathbb{N}.$ Let us assume that $M_0\leq {\eta}^{\frac{r}{k_r}}V_{\eta,r}(\mu)$
		for all $\eta<n$ and $n\geq n_0.$ Thus, we have
		$$V_{n,r}(\mu)\geq  \sum_{i=1}^{N}p_i s_{i}^{r} V_{k_{i}(n),r}(\mu)\geq M_0 \sum_{i=1}^{N}p_i s_{i}^{r} {k_i(n)}^{\frac{-r}{k_r}}.$$
		Hence, by the generalized H\"older's inequality, we get
		$$\sum_{i=1}^{N}p_i s_{i}^{r} {k_i(n)}^{\frac{-r}{k_r}}\geq \bigg(\sum_{i=1}^{N}(p_i s_{i}^{r})^{\frac{k_r}{r+k_r}}\bigg)^{1+\frac{r}{k_r}}\cdot \bigg(\sum_{i=1}^{N}k_i(n)\bigg)^{\frac{-r}{k_r}}.$$
		Using the facts that $\sum_{i=1}^{N}(p_i s_{i}^{r})^{\frac{k_r}{r+k_r}}=1$ and $\sum_{i=1}^{N}k_i(n)\leq n$, we get $V_{n,r}(\mu)\geq M_0 n^{\frac{-r}{k_r}}$, that is, $M_0\leq n^{\frac{r}{k_r}}V_{n,r}(\mu).$ Thus, by induction, we obtain that $M_0\leq n^{\frac{r}{k_r}}V_{n,r}(\mu)$ holds for all $n\in \mathbb{N}.$ This implies that
		$$\liminf_{n\to \infty} n V_{n,r}^{\frac{k_r}{r}}(\mu)\geq M_{0}^{\frac{k_r}{r}}>0.$$
		Thus, the proof is done.
	\end{proof}
	For determining the upper bound of the upper quantization dimension, we assume that the measure $\nu$ is an invariant measure corresponding to the IFS $\mathcal{G}=\{\mathbb{R}^d; g_1,g_2,\ldots,g_M\}$ and probability vector $(t_1,t_2,\ldots,t_M)$. Thus, we have 	$$\nu=\sum_{i=1}^{M}t_i\nu\circ g_{i}^{-1},\quad  C=\bigcup_{i=1}^{M}g_i(C), \quad \nu(C)=1.$$
	Motivated by \cite{Z5}, we fix some notations in the following paragraph.
	
	First, we define a finite maximal antichain on $I^*$ as follows:
	$$\Gamma_{n}^{r}=\{w\in I^*: p_w c_{w}^{r}< \xi_{r}^{n}\leq p_{w^-} c_{w^-}^{r} \}$$
	where $n\geq 1$ and $\xi_{r}=\min\{\min\limits_{1\leq j\leq M}t_jr_{j}^{r}, \min\limits_{1\leq j\leq N}p_jc_{j}^{r} \}.$ Let $|\Gamma_{n}^{r}|$ denote the
	cardinality of the set $\Gamma_{n}^{r}.$  We write $$\Phi_{n}^{r}= \bigcup_{m=0}^{l_{\Gamma_{n}^{r}}-1}I^m \cup \Delta_{\Gamma_{n}^{r}}^{*}$$
	for all $n\geq 1.$ By the definition of $\Phi_{n}^{r}$, one can easily see that for any $w\in \Phi_{n}^{r}$, we have $p_w c_{w}^{r}\geq \xi_{r}^{n}.$
	Thus, for any $w\in \Phi_{n}^{r},$ we define a finite maximal antichain on $J^*$ as follows
	$$\Gamma_{n}^{r}(w)=\{\sigma\in J^*: p_w c_{w}^{r} t_{\sigma}r_{\sigma}^{r}< \xi_{r}^{n}\leq p_{w} c_{w}^{r} t_{\sigma^{-}}r_{\sigma^{-}}^{r}\}.$$
	 For each $w\in \Phi_{n}^{r}$, let $|\Gamma_{n}^{r}(w)|$ be the cardinality of $\Gamma_{n}^{r}(w)$.  We define
	\begin{align}\label{eq3}
		\phi_{n,r}:= |\Gamma_{n}^{r}|+ \sum_{w\in \Phi_{n}^{r}} |\Gamma_{n}^{r}(w)|.
	\end{align}
	\begin{lemma} \label{lem3.3}
		For each $n\geq 1$, we have
		$$K=\bigg(\bigcup_{w\in \Gamma_{n}^{r}}S_{w}(K)\bigg)\cup \bigg(\bigcup_{w\in \Phi_{n}^{r}}\bigcup_{\sigma\in \Gamma_{n}^{r}(w)}S_{w}(g_\sigma(C))\bigg).$$
	\end{lemma}
	\begin{proof}
		By Lemma \ref{Z}, we have
		$$K=\bigg(\bigcup_{w\in \Gamma_{n}^{r}}S_{w}(K)\bigg)\cup \bigg(\bigcup_{w\in \Delta_{\Gamma_{n}^{r}}^{*}} S_{w}(C)\bigg)\cup \bigg(\bigcup_{m=0}^{l_{\Gamma_{n}^{r}}-1}\bigcup_{w\in I^m}S_{w}(C)\bigg).$$
		Since for each $w\in \Phi_{n}^{r} $, $\Gamma_{n}^{r}(w)$ is a finite maximal antichain in $J^*$, we obtian
		$$C=\bigcup_{\sigma\in \Gamma_{n}^{r}(w)}g_{\sigma}(C).$$
		Using the definition of $\Phi_{n}^{r}$ and the above relations, we get
		$$K=\bigg(\bigcup_{w\in \Gamma_{n}^{r}}S_{w}(K)\bigg)\cup \bigg(\bigcup_{w\in \Phi_{n}^{r}}\bigcup_{\sigma\in \Gamma_{n}^{r}(w)}S_{w}(g_\sigma(C))\bigg).$$
	\end{proof}
	\begin{lemma}\label{lem3.4}
		Assume that the condensation system $(\{S_i\}_{i=1}^{N},(p_i)_{i=0}^{N},\nu)$ satisfies the strong separation condition and the IFS $\mathcal{G}$ satisfies the strong open set condition. Let $\mu$ be the inhomogeneous measure associated with   $(\{S_i\}_{i=1}^{N},(p_i)_{i=0}^{N},\nu)$. Then, we have
		$$V_{\phi_{n,r},r}(\mu)\leq C^*\bigg(\sum_{w\in \Gamma_{n}^{r}} p_w c_{w}^{r} + \sum_{w\in \Phi_{n}^{r}}\sum_{\sigma\in \Gamma_{n}^{r}(w)} p_wc_{w}^{r} t_\sigma r_{\sigma}^{r}\bigg),$$ 	where $C^*=\max\{({\text{diam}(K)})^r, ({\text{diam}(C)})^r  \}$ and $\phi_{n,r}$ is given by equation \eqref{eq3}.
	\end{lemma}
	\begin{proof}
		First, we define a set $A$ such that $\text{card}(A)\leq \phi_{n,r}.$ For each $w\in \Gamma_{n}^{r},$ we choose an element $a_{w}\in S_{w}(K)$ and for each $w\in \Phi_{n}^{r}$, we choose an element $a_{w,\sigma}\in S_{w}(g_\sigma(C))$ such that $\sigma \in \Gamma_{n}^{r}(w).$ We collect all such type of elements and form a set $A$. Clearly,  $\text{card}(A)\leq \phi_{n,r}.$ Thus, by Lemma \ref{lem3.3}, we have
		\begin{align*}
			V_{\phi_{n,r},r}(\mu)&\leq \int_{K} d(x,A)^r d\mu(x) \\
			&\leq \sum_{w\in \Gamma_{n}^{r}}\int_{S_{w}(K)} d(x,A)^r d\mu(x)+ \sum_{w\in \Phi_{n}^{r}}\sum_{\sigma\in \Gamma_{n}^{r}(w)}\int_{S_{w}(g_\sigma(C))} d(x,A)^r d\mu(x)\\
			&\leq \sum_{w\in \Gamma_{n}^{r}}\int_{S_{w}(K)} d(x,a_{w})^r d\mu(x)+ \sum_{w\in \Phi_{n}^{r}}\sum_{\sigma\in \Gamma_{n}^{r}(w)}\int_{S_{w}(g_\sigma(C))} d(x,a_{w,\sigma})^r d\mu(x)\\
			&\leq \sum_{w\in \Gamma_{n}^{r}} (\text{diam}(S_{w}(K)))^r \mu(S_{w}(K))+ \sum_{w\in \Phi_{n}^{r}}\sum_{\sigma\in \Gamma_{n}^{r}(w)}(\text{diam}(S_{w}(g_\sigma(C))))^r  \mu(S_{w}(g_\sigma(C)))\\
			&\leq  \sum_{w\in \Gamma_{n}^{r}} c_{w}^{r}({\text{diam}(K)})^r \mu(S_{w}(K))+ \sum_{w\in \Phi_{n}^{r}}\sum_{\sigma\in \Gamma_{n}^{r}(w)} c_{w}^{r} r_{\sigma}^{r}({\text{diam}(C)})^r \mu(S_{w}(g_\sigma(C))).
		\end{align*}
		Writing $C^*=\max\{({\text{diam}(K)})^r, ({\text{diam}(C)})^r  \}$, we get
		\begin{equation}\label{eq4}
			V_{\phi_{n,r},r}(\mu)\leq C^*\bigg(\sum_{w\in \Gamma_{n}^{r}} c_{w}^{r} \mu(S_{w}(K))+ \sum_{w\in \Phi_{n}^{r}}\sum_{\sigma\in \Gamma_{n}^{r}(w)} c_{w}^{r} r_{\sigma}^{r} \mu(S_{w}(g_\sigma(C)))\bigg).
		\end{equation}
		Now, we determine the values of $\mu(S_{w}(K))$ and
		$\mu(S_{w}(g_\sigma(C)))$ for $w\in \Gamma_{n}^{r}$ and $\sigma\in \Gamma_{n}^{r}(w)$. Let $w\in \Gamma_{n}^{r}$ and $|w|=n_1$. Since
		$\mu=\sum_{\tau\in I^{n_1}}p_\tau\mu\circ S_{\tau}^{-1}+p_0\sum_{m=0}^{n_{1}-1}\sum_{\tau\in I^m}p_\tau \nu \circ S_{\tau}^{-1}$, and the condensation system satisfies the strong separation condition, we have
		\begin{align*}
			\mu(S_{w}(K))&=\sum_{\tau\in I^{n_1}}p_\tau\mu\circ S_{\tau}^{-1}(S_{w}(K))+p_0\sum_{m=0}^{n_{1}-1}\sum_{\tau \in I^m}p_\tau \nu \circ S_{\tau}^{-1}(S_{w}(K))\\&= p_w\mu( S_{w}^{-1}(S_{w}(K)))+	p_0\sum_{m=0}^{n_{1}-1}\sum_{\tau\in I^m}p_\tau \nu \circ S_{\tau}^{-1}(S_{w}(K))\\&= p_w\mu(K)+ p_0\sum_{m=0}^{n_{1}-1}\sum_{\substack{\tau\in I^m\\ \tau\prec w}}p_\tau \nu \circ S_{\tau}^{-1}(S_{w}(K)).
		\end{align*}
		Since $\tau\in I^m$ and $\tau\prec w$, we write $w=\tau w_\tau$ for some $w_\tau\in I^*$. Then, using $\mu(K)=1$, we have
		\begin{align*}
			\mu(S_{w}(K))&= p_w+ p_0\sum_{m=0}^{n_{1}-1}\sum_{\substack{\tau\in I^m\\ \tau\ \prec w}}p_\tau \nu (S_{\tau}^{-1}(S_{\tau w_\tau}(K)))\\&=p_w+ p_0\sum_{m=0}^{n_{1}-1}\sum_{\substack{\tau\in I^m\\ \tau \prec w}}p_\tau \nu (S_{w_\tau}(K)\cap C).
		\end{align*}
		Since the condensation system satisfies the strong separation condition, we have $S_{w_\tau}(K)\cap C=\emptyset$. Thus, we get
		\begin{equation}\label{eq5}
			\mu(S_{w}(K))= p_w.
		\end{equation}
		Since the IFS $J$ satisfies the strong open set condition, we have $\nu(g_\sigma(C))=t_\sigma$ for all $\sigma\in J^*$ (see \cite{BBHV, M}).
		\begin{align*}
			\mu(S_{w}(g_\sigma(C)))&=\sum_{\tau\in I^{n_1}}p_\tau\mu\circ S_{\tau}^{-1}(S_{w}(g_\sigma(C)))+p_0\sum_{m=0}^{n_{1}-1}\sum_{\tau\in I^m}p_\tau \nu \circ S_{\tau}^{-1}(S_{w}(g_\sigma(C)))\\&= p_w\mu( S_{w}^{-1}(S_{w}(g_\sigma(C))))+	p_0\sum_{m=0}^{n_{1}-1}\sum_{\substack{\tau\in I^m\\ \tau\prec w}}p_\tau \nu \circ S_{\tau}^{-1}(S_{w}(g_\sigma(C)))\\&= p_w\mu(g_\sigma(C))+0\\&=p_w \bigg(p_0\nu(g_\sigma(C))+ \sum_{i=1}^{N}p_i\mu\circ S_{i}^{-1}(g_\sigma(C))\bigg).
		\end{align*}
		Since the condensation system satisfies the strong separation condition, we obtain
		\begin{equation}\label{eq6}
			\mu(S_{w}(g_\sigma(C)))=p_w p_0\nu(g_\sigma(C))=p_w p_0 t_{\sigma}\leq p_{w} t_{\sigma}.
		\end{equation}
		By using \eqref{eq5} and \eqref{eq6} in the inequality \eqref{eq4}, we get the required result.
	\end{proof}
	\begin{prop}\label{Prop3.5}
		Assume that the condensation system $(\{S_i\}_{i=1}^{N},(p_i)_{i=0}^{N},\nu)$ satisfies the strong separation condition and the IFS $\mathcal{G}$ satisfies the strong open set condition. Let $\mu$ be the inhomogeneous measure associated with   $(\{S_i\}_{i=1}^{N},(p_i)_{i=0}^{N},\nu)$. Then, $$\limsup_{n\to \infty}n V_{n,r}^{\frac{t}{r}}(\mu)<\infty,$$
		for all $t>\max\{l_r,d_r\},$ where $l_r$ and $d_r$ are given by $\sum_{i=1}^{N}(p_ic_{i}^{r})^{\frac{l_r}{r+l_r}}=1$ and $\sum_{i=1}^{M}(t_ir_{i}^{r})^{\frac{d_r}{r+d_r}}=1$, respectively. 
	\end{prop}	
	\begin{proof}
		By Lemma \ref{lem3.4} and using the definition of $\Gamma_{n}^{r}$ and $\Gamma_{n}^{r}(w)$, we have
		\begin{align*}
			V_{\phi_{n,r},r}(\mu)&\leq C^*\bigg(\sum_{w\in \Gamma_{n}^{r}} p_w c_{w}^{r} + \sum_{w\in \Phi_{n}^{r}}\sum_{\sigma\in \Gamma_{n}^{r}(w)} p_wc_{w}^{r} t_\sigma r_{\sigma}^{r}\bigg)\\&\leq C^* \xi_{r}^{n} \phi_{n,r}.
		\end{align*}
		Now, we determine the bound of $\phi_{n,r}$. Suppose the values of $l_r$ and $d_r$ are uniquely determined by
		$\sum_{i=1}^{N}(p_ic_{i}^{r})^{\frac{l_r}{r+l_r}}=1$ and $\sum_{i=1}^{M}(t_ir_{i}^{r})^{\frac{d_r}{r+d_r}}=1$, respectively. Since $\Gamma_{n}^{r}$ is a finite maximal antichain in $I^*$  and for each $w\in \Phi_{n}^{r} $, $\Gamma_{n}^{r}(w)$ is a  finite maximal antichain in $J^*$, we have
		$$\sum_{w\in \Gamma_{n}^{r}}(p_w c_{w}^{r})^{\frac{l_r}{r+l_r}}=1 \quad \text{and} \quad  \sum_{\sigma\in \Gamma_{n}^{r}(w)}(t_{\sigma}r_{\sigma}^{r})^{\frac{d_r}{r+d_r}}=1.$$
		Let $t>\max\{l_r,d_r\}.$ Then, we have
		$$\sum_{w\in \Gamma_{n}^{r}}(p_w c_{w}^{r})^{\frac{t}{r+t}}<1 \quad \text{and} \quad  \sum_{\sigma\in \Gamma_{n}^{r}(w)}(t_{\sigma}r_{\sigma}^{r})^{\frac{t}{r+t}}<1.$$
		By the definitions of $\phi_{n,r}$, $\Gamma_{n}^{r}$
		and $\Gamma_{n}^{r}(w)$ for each $w\in \Phi_{n}^{r}$, we obtain
		\begin{align*}
			\phi_{n,r}\xi_{r}^{\frac{(n+1)t}{r+t}}&\leq \sum_{w\in \Gamma_{n}^{r}} (p_w c_{w}^{r})^{\frac{t}{r+t}} + \sum_{w\in \Phi_{n}^{r}}\sum_{\sigma\in \Gamma_{n}^{r}(w)} (p_wc_{w}^{r} t_\sigma r_{\sigma}^{r})^{\frac{t}{r+t}}\\& < 1+ \sum_{w\in \Phi_{n}^{r}}(p_wc_{w}^{r})^{\frac{t}{r+t}}\sum_{\sigma\in \Gamma_{n}^{r}(w)}  (t_\sigma r_{\sigma}^{r})^{\frac{t}{r+t}}\\&< 1+ \sum_{w\in \Phi_{n}^{r}}(p_wc_{w}^{r})^{\frac{t}{r+t}}.
		\end{align*}
		Since $\Phi_{n}^{r}= \bigcup_{m=0}^{l_{\Gamma_{n}^{r}}-1}I^m \cup \Delta_{\Gamma_{n}^{r}}^{*}$
		and $\Delta_{\Gamma_{n}^{r}}^{*}\subset \bigcup_{m=l_{\Gamma_{n}^{r}}}^{M_{\Gamma_{n}^{r}}} I^m$, we have
		\begin{align*}
			\phi_{n,r}\xi_{r}^{\frac{(n+1)t}{r+t}}&< 1+\sum_{m=0}^{l_{\Gamma_{n}^{r}}-1} \sum_{w\in I^m}(p_wc_{w}^{r})^{\frac{t}{r+t}}+ \sum_{m=l_{\Gamma_{n}^{r}}}^{M_{\Gamma_{n}^{r}}} \sum_{w\in I^m}(p_wc_{w}^{r})^{\frac{t}{r+t}}\\&= 1+\sum_{m=0}^{M_{\Gamma_{n}^{r}}} \sum_{w\in I^m}(p_wc_{w}^{r})^{\frac{t}{r+t}}\\&= 1+\sum_{m=0}^{M_{\Gamma_{n}^{r}}} \bigg(\sum_{i=1}^{N}(p_ic_{i}^{r})^{\frac{t}{r+t}}\bigg)^m\\&< \frac{2}{1-\sum_{i=1}^{N}(p_ic_{i}^{r})^{\frac{t}{r+t}}}.
		\end{align*}
		Thus, by the previous inequalities, we obtain
		\begin{align*}
			\phi_{n,r} V_{\phi_{n,r},r}^{\frac{t}{r}}(\mu)&\leq  {(C^*)}^{\frac{t}{r}} \xi_{r}^{\frac{nt}{r}} \phi_{n,r}^{\frac{t+r}{r}}\\&\leq {(C^*)}^{\frac{t}{r}} \xi_{r}^{\frac{-t}{r}}\bigg(\frac{2}{1-\sum_{i=1}^{N}(p_ic_{i}^{r})^{\frac{t}{r+t}}}\bigg)^{\frac{t+r}{r}}.
		\end{align*}
		This implies that $\limsup_{n\to \infty}\phi_{n,r} V_{\phi_{n,r},r}^{\frac{t}{r}}<\infty$ for all $t> \max\{l_r,d_r\}$. By using \cite[Theorem 2.4]{Z4}, we obtain $$\limsup_{n\to \infty}n V_{n,r}^{\frac{t}{r}}(\mu)<\infty \quad \forall \quad  t>\max\{l_r,d_r\}.$$
		This completes the proof.
		
	\end{proof}	
	\subsection*{Proof of Theorem~\ref{main}} Let $\mu$ be the condensation measure associated with the condensation system
	$(\{S_i\}_{i=1}^{N}, (p_i)_{i=0}^{N},\nu)$. Assume that the condensation system $(\{S_i\}_{i=1}^{N},(p_i)_{i=0}^{N},\nu)$ satisfies the strong separation condition. Then, by Lemma \ref{lem3.1}, Proposition \ref{Prop3.2} and $(2)$ of Proposition \ref{Prop2.2}, we obtain
	$$\max\{k_r, \underline{D}_r(\nu)\}\leq \underline{D}_r(\mu),$$
	where $k_r$ is uniquely determined by $\sum_{i=1}^{N}(p_is_{i}^{r})^{\frac{k_r}{r+k_r}}=1.$\\ Furthermore, if the measure $\nu$ is an invariant measure corresponding to the IFS $\mathcal{G}$ with probability vector $(t_1,t_2,\ldots,t_M)$ and  the IFS $\mathcal{G}$ satisfies the strong open set condition, then by  Proposition \ref{Prop3.5} and $(1)$ of Proposition \ref{Prop2.2}, we get
	$$\overline{D}_r(\mu)\leq \max\{l_r,d_r\},$$
	where $l_r$ and $d_r$ are given by $\sum_{i=1}^{N}(p_ic_{i}^{r})^{\frac{l_r}{r+l_r}}=1$ and $\sum_{i=1}^{M}(t_ir_{i}^{r})^{\frac{d_r}{r+d_r}}=1$, respectively.
	
	\qed
	
	\begin{remark}
		Let $\mu$ be the condensation measure for a condensation system
		$(\{S_i\}_{i=1}^{N}, (p_i)_{i=0}^{N},\nu)$, where $\nu$ is any Borel probability measure supported on any compact set $C$. Then, by Lemma \ref{lem3.1} and  Proposition \ref{Prop3.2}, the lower bound obtained in Theorem \ref{main} is still valid in this more general setting. The assumption that the measure $\nu$ is the invariant measure corresponding to some IFS is needed only for the upper bound of the quantization dimension.
	\end{remark}	

\noindent {\bf Statements and Declarations:}
\begin{itemize}
	\item {\bf Competing Interests:} The authors have no relevant financial or non-financial interests
	to disclose.
	\item {\bf Data availability:} Data sharing is not applicable to this article as no datasets were generated or analysed during the current study.
	\item {\bf Funding:} The authors declare that no funds, grants, or other support were received
	during the preparation of this manuscript.
	\item 	{\bf Author Contributions:} All authors contributed equally in this manuscript.
\end{itemize}

\noindent {\bf Acknowledgements:} We would like to thank the anonymous referees for their valuable comments and
suggestions.


\begin{thebibliography}{9999}

\bibitem{AW} E. F. Abaya and G. L. Wise, \emph{On the existence of optimal quantizers}, IEEE Trans. Inform. Theory {\bf 28} (1982), no. 6, 937--946.

 \bibitem{ALSW} R. Achour, Z. Li, B. Selmi, and T. Wang, \emph{General fractal dimensions of graphs of products and sums of
continuous functions and their decompositions},  J. Math. Anal. Appl. {\bf 538} (2024), no. 2, 128400.

\bibitem{AS} R. Achour, and B. Selmi, \emph{General fractal dimensions of typical sets and measures}, Fuzzy Sets and Systems {\bf 490} (2024) 109039. 

\bibitem{BBHV} C. Bandt, M. Barnsley, M. Hegland and A. Vince, \emph{Old wine in fractal bottles I: Orthogonal expansions on self-referential spaces via fractal transformations}, Chaos Solitons Fractals {\bf 91} (2016), 478--489.	
    
 \bibitem{BD} M. F. Barnsley and S. Demko, \emph{Iterated function systems and the global construction of fractals}, Proc. R. Soc. Lond. Ser. A, {\bf 399} (1985), no. 1817, 243--275.
			
\bibitem{BCDRV} P. Biteng, M. Caguiat, D. Deb,  M.K. Roychowdhury and B. Villanueva, \emph{Constrained quantization for a uniform distribution}, to appear, Houston Journal of Mathematics.

\bibitem{B} J. A. Bucklew, \emph{Two results on the asymptotic performance of quantizers}, IEEE Trans. Inform. Theory {\bf 30} (1984), no. 2, 341--348.

\bibitem{BW} J. A. Bucklew and G. L. Wise, \emph{Multidimensional asymptotic quantization theory with $r$th power distortion measures}, IEEE Trans. Inform. Theory {\bf 28} (1982), no. 2, 239--247.
		
\bibitem{DAi} M. Dai and Z. Liu, \tit{The quantization dimension and other dimensions of probability measures}, Int. J. Nonlinear Sci. {\bf 5} (2008), no. 3, 267--274.
		
\bibitem{GG} A. Gersho and R. M. Gray, \emph{Vector Quantization and Signal Compression}, Kluwer Academic Publishers, Boston, 1991.
\bibitem{GL1} S. Graf and H. Luschgy, \emph{Foundations of Quantization for Probability Distributions}, Lecture Notes in Mathematics {\bf 1730}, Springer, Berlin, 2000.
		
\bibitem{GL2} S. Graf and H. Luschgy, \emph{Asymptotics of the Quantization Errors for Self-Similar Probabilities}, Real Anal. Exchange {\bf 26} (2000), no. 2, 795--810.		
		
		
\bibitem{GL3} S. Graf and H. Luschgy, \emph{The quantization dimension of self-similar probabilities}, Math. Nachr. {\bf 241} (2002), no. 1, 103--109.
		
\bibitem{GL4} S. Graf and H. Luschgy, \emph{Quantization for probability measures with respect to the geometric mean error}, Math. Proc. Cambridge Philos. Soc. {\bf 136} (2004), no. 3, 687--717.
		
\bibitem{GKL}  R. M. Gray, J. C. Kieffer and Y. Linde, \emph{Locally optimal block quantizer design}, Inform. and Control {\bf 45} (1980), no. 2, 178--198.	
	
\bibitem{GN} R. Gray and D. Neuhoff, \emph{Quantization}, IEEE. Trans. Inform Theory, {\bf 44} (1998), no. 6, 2325--2383.
			
\bibitem{G}     P. M. Gruber, \emph{Optimum quantization and its applications}, Adv. Math. {\bf 186} (2004), no. 2, 456--497.
		
\bibitem{GL} A. Gy\"orgy and T. Linder, \emph{On the structure of optimal entropy-constrained scalar quantizers},  IEEE Trans. Inform. Theory {\bf 48} (2002), no.2, 416--427.
		
\bibitem{L} A. Lasota, \emph{A variational principle for fractal dimensions}, Nonlinear Anal. {\bf 64} (2006), no. 3, 618--628.

\bibitem {LM} L. J. Lindsay and R. D. Mauldin, \emph{Quantization dimension for conformal iterated function systems}, Nonlinearity {\bf 15} (2002), no. 1, 189--199.
 
\bibitem {MR} E. Mihailescu and M. K. Roychowdhury, \emph{Quantization coefficients in infinite systems}, Kyoto J. Math. {\bf 55} (2015), no. 4, 857--873.

\bibitem{M} M. Mor{\'a}n, \emph{Dynamical boundary of a self-similar set}, Fund. Math. {\bf 160} (1999), no. 1, 1--14.

\bibitem {PR1} M. Pandey and M. K. Roychowdhury, \emph{Constrained quantization for the Cantor distribution}, J. Fractal Geom. {\bf 11} (2024), no. 3/4, pp. 319-341.

\bibitem {PR2} M. Pandey and M. K. Roychowdhury, Conditional constrained and unconstrained quantization for probability distributions, https://arxiv.org/abs/2312.02965.
		
\bibitem{PRSS} Y. Peres, M. Rams, K. Simon and B. Solomyak, \emph{Equivalence of positive Hausdorff measure and the open set condition for self-conformal sets}, Proc. Amer. Math. Soc. {\bf 129} (2001), no. 9, 2689--2699.

\bibitem{Potz} K. P\"otzelberger, \tit{The quantization dimension of distributions}, Math. Proc. Cambridge Philos. Soc. {\bf 131} (2001), no. 3, 507--519.

\bibitem {PRV} A. Priyadarshi, M. K. Roychowdhury and M. Verma, \emph{Quantization dimensions for the bi-Lipschitz recurrent iterated function systems}, Dyn. Syst. (2024), 1-20 https://doi.org/10.1080/14689367.2024.2424226
 
\bibitem {PVV}    A. Priyadarshi, M. Verma and S. Verma, \emph{Fractal dimensions of fractal transformations and quantization dimensions for bi-Lipschitz mappings}, J. Fractal Geom. (2024), published online first, https://doi.org/10.4171/jfg/157	
    
\bibitem{R1} M. K. Roychowdhury, \emph{Quantization dimension function and ergodic measure with bounded distortion}, Bull. Pol. Acad. Sci. Math. {\bf 57} (2009), no. 3-4, 251--262.
 
\bibitem{R2} M. K. Roychowdhury, \emph{Quantization dimension for some Moran measures}, Proc. Amer. Math. Soc. {\bf 138} (2010), no. 11,  4045--4057.
		
\bibitem{R3} M. K. Roychowdhury, \emph{Quantization dimension and temperature function for recurrent self-similar measures}, Chaos Solitons Fractals, {\bf 44} (2011), no. 11, 947--953.

\bibitem{R4} M. K. Roychowdhury,  \emph{Quantization dimension function and Gibbs measure associated with Moran set}, J. Math. Anal. Appl. {\bf 373} (2011), no. 1, 73--82.

\bibitem{R5} M. K. Roychowdhury, \emph{Quantization dimension and temperature function for bi-Lipschitz mappings}, Israel J. Math. {\bf 192} (2012), no. 1, 473--488.

\bibitem{R} M. K. Roychowdhury, \emph{Quantization dimension estimate of inhomogeneous self-similar measures}, Bull. Pol. Acad. Sci. Math. {\bf 61} (2013), no. 1, 35--45.
		
\bibitem{R6} M. K. Roychowdhury, \emph{Topological pressure and fractal dimensions of cookie-cutter-like sets}, Southeast Asian Bull. Math. {\bf 44} (2020), no.  5, 719--732.  		

 \bibitem{RS} M. K. Roychowdhury and B. Selmi, \emph{Local dimensions and quantization dimensions in dynamical systems}, J. Geom. Anal. {\bf 31} (2021) no. 6, 6387--6409.
 
\bibitem{S} A. Schief, \emph{Separation properties for self-similar sets}, Proc. Amer. Math. Soc. {\bf 122} (1994), no. 1, 111--115.

\bibitem{S1} B. Selmi, \emph{General multifractal dimensions of measures}, Fuzzy Sets and Systems, {\bf 499} (2025), 109177. 

\bibitem{SW} T. Szarek and S. L. Wedrychowicz, \emph{The OSC does not imply the SOSC for infinite iterated function systems}, Proc. Amer. Math. Soc. {\bf 133} (2005), no. 2, 437--440.
		
\bibitem{Z1} P. L. Zador, \emph{Development and evaluation of procedures for quantizing multivariate distributions}, Thesis (Ph.D.)--Stanford University, 1964.
		
\bibitem{Z2} P. L. Zador, \emph{Asymptotic quantization error of continuous signals and the quantization dimension}, IEEE Trans. Inform. Theory {\bf 28} (1982), no. 2, 139--149.

\bibitem{Z3} S. Zhu, \emph{Quantization dimension for condensation systems}, Math. Z. {\bf 259} (2008), no. 1, 33--43.

\bibitem{Z4} S. Zhu, \emph{On the upper and lower quantization coefficient for probability measures on multiscale Moran sets}, Chaos Solitons Fractals {\bf 45} (2012), no. 11, 1437--1443.

\bibitem{Z5} S. Zhu, \emph{The quantization for in-homogeneous self-similar measures with in-homogeneous open set condition}, Internat. J. Math. {\bf 26} (2015), no. 5, 1550030.
 
	\end{thebibliography}
\end{document}